\documentclass[a4paper]{cas-sc}
\usepackage[authoryear,longnamesfirst]{natbib}

\usepackage[utf8]{inputenc} 
\usepackage[T1]{fontenc}    

\usepackage{hyperref}
\usepackage{booktabs}       
\usepackage{nicefrac}       
\usepackage{microtype}      
\usepackage{xcolor}         

\usepackage{amsfonts,amsthm,array}

\usepackage{algorithm}
\usepackage[font=small]{caption}
\usepackage{subcaption}
\usepackage{color}
\usepackage{graphicx}
\usepackage{enumitem}
\usepackage{bbm}
\usepackage{threeparttable}
\usepackage{algorithmicx}
\usepackage{algpseudocode}

\usepackage{natbib}

\usepackage{mathtools}
\usepackage{latexsym}
\usepackage{epsfig}

\theoremstyle{plain}

\newtheorem{theorem}{Theorem}

\newtheorem{assumption}{Assumption}
\newtheorem{lemma}{Lemma}

\newtheorem{example}{Example}

\usepackage{dsfont}
\usepackage{pifont}

\usepackage{xspace}

\usepackage{amsmath}
\usepackage{amssymb}
\usepackage{bm}

\def\tsc#1{\csdef{#1}{\textsc{\lowercase{#1}}\xspace}}
\tsc{WGM}
\tsc{QE}
\tsc{EP}
\tsc{PMS}
\tsc{BEC}
\tsc{DE}

\begin{document}

\let\WriteBookmarks\relax
\def\floatpagepagefraction{1}
\def\textpagefraction{.001}

\shorttitle{Universal methods for  variational inequalities}

\shortauthors{A.~Klimza et al.}

\title [mode = title]{Universal methods for  variational inequalities: deterministic and stochastic cases}                      



%
\author[1]{Anton Klimza}[orcid=0009-0001-6839-0194]
\ead{klimza.aa@phystech.edu}

\author[2, 1, 3]{Alexander Gasnikov}[orcid=0000-0002-7386-039X]
\ead{gasnikov@yandex.ru}

\author[1, 4]{Fedor Stonyakin}[orcid=0000-0002-9250-4438]
\ead{fedyor@mail.ru}

\author[1]{Mohammad Alkousa}[orcid=0000-0001-5470-0182]
\ead{mohammad.alkousa@phystech.edu}

\affiliation[1]{organization={Moscow Institute of Physics and Technology, Department of Applied Mathematics and Computer Science},
    addressline={Institutsky lane 9}, 
    city={Dolgoprudny},
    postcode={141701}, 
    country={Russian Federation}}

\affiliation[2]{organization={Innopolis University},
    addressline={Universitetskaya Str.}, 
    city={Innopolis},
    postcode={420500}, 
    country={Russian Federation}}

\affiliation[3]{organization={Institute for Information Transmission Problems RAS},
    addressline={Bolshoy Karetny per. 19, build.1}, 
    city={Moscow},
    postcode={127051}, 
    country={Russian Federation}}

\affiliation[4]{organization={V.\,I.\,Vernadsky Crimean Federal University},
    addressline={4 Academician Vernadsky Avenue}, 
    city={Simferopol},
    postcode={295007}, 
    country={ Republic of Crimea, Russian Federation}}


\begin{abstract}
In this paper, we propose universal proximal mirror methods to solve the variational inequality problem with  H\"older-continuous operators in both deterministic and stochastic settings. The proposed methods automatically adapt not only to the oracle’s noise (in the stochastic setting of the problem) but also to the H\"older continuity of the operator without having prior knowledge of either the problem class or the nature of the operator information. We analyzed the proposed algorithms in both deterministic and stochastic settings and obtained estimates for the required number of iterations to achieve a given quality of a solution to the variational inequality. We showed that, without knowing the H\"older exponent and H\"older constant of the operators, the proposed algorithms have the least possible in the worst-case sense complexity for the considered class of variational inequalities. We also compared the resulting stochastic algorithm with other popular optimizers for the task of image classification.
\end{abstract}



\begin{keywords}
Variational inequality \sep H\"older continuous operator \sep Universal method \sep Proximal mirror method \sep Saddle point problem \sep Image classification
\end{keywords}

\maketitle

\section{Introduction} \label{sec:Introduction}

Variational inequalities (VIs) often arise in a wide variety of mathematical problems, such as optimization, optimal control, partial differential equations, mechanics, and finance. For a comprehensive overview, please refer to \cite{konnov2007equilibrium,facchinei2003finite,elliott1987variational,nagurney1993network}. Variational inequalities cover as a special case many optimization problems such as minimization problems, saddle point problems, and fixed point problem (see Examples \ref{ex:minproblem}, \ref{ex:saddleproblem} and \ref{ex:fixedproblem} below). They play a key role in solving equilibrium and complementarity problems \cite{facchinei2003finite,harker1990finite}, and have many applications in smooth and non-smooth optimization problems \cite{nesterov2005smooth}, also in saddle point problems that are an important part of the machine learning research \cite{arjovsky2017bottou,kniaz2021adversarial}, generative adversarial networks \cite{goodfellow2020generative}, supervised/unsupervised learning \cite{bach2012optimization,bach2008convex,esser2010general,joachims2005support,xu2004maximum}, reinforcement learning \cite{jin2020efficiently,omidshafiei2017deep}, adversarial training \cite{madry2017towards}, and generative models \cite{daskalakis2017training,goodfellow2014generative,gidel2018variational}.

Numerous researchers have dedicated their efforts to exploring theoretical aspects related to the existence and stability of solutions and constructing iterative methods for solving VIs. A significant contribution to the development of numerical methods for solving VIs was made in the 1970’s, when the extragradient method was proposed in \cite{korpelevich1976extragradient}. More recently, Nemirovski in his seminal work \cite{nemirovski2004prox} proposed a non-Euclidean variant of this method, called Mirror Prox algorithm, which can be applied to Lipschitz continuous operators.   Different methods with similar complexity were also proposed in \cite{auslender2005interior,gasnikov2019adaptive,monteiro2010complexity,nesterov2007dual,solodov1999hybrid}. Besides that, in \cite{nesterov2007dual}, Nesterov proposed a method for variational inequalities with a bounded variation of the operator, i.e., with a non-smooth operator. He also raised the question of whether it is possible to propose a method that automatically “adjusts to the actual level of smoothness of the current problem instance.” One of the goals of this paper is to propose such an algorithm. There is also extensive
literature on variations of extragradient method that avoid taking two steps or two gradient computations per iteration, and so on (see for example \cite{hsieh2019convergence,malitsky2020forward}). 

Furthermore, the rising significance of stochastic methods in large-scale computations has driven interest in stochastic extragradient methods. In \cite{juditsky2011solving}, the authors studied stochastic Mirror-Prox methods for VIs with compact convex feasible sets, and in \cite{mishchenko2020revisiting} it improved the stochastic Mirror-Prox method for VIs by using a single sample per iteration. In \cite{hsieh2019convergence}, the authors showed stochastic variants of the single-call type extragradient methods. Yet, with a constant stepsize, the iterates of these methods only converge to a neighborhood of the solution set. Diminishing stepsizes do ensure convergence, but they empirically slow down the performance. For more information about stochastic VIs, see \cite{beznosikov2023smooth} and references therein.

Similarly to a more general class of optimization problems with so-called H\"{o}lder-continuous functions, which cover both smooth and non-smooth problems, and for which authors proposed so-called universal algorithms that do not use the information about the H\"{o}lder exponent and H\"{o}lder constant \cite{kamzolov2021universal,baimurzina2019universal,ghadimi2019generalized,gasnikov2018universal,nesterov2015universal} (the main advantage of the universal method is that it adjusts itself to the smoothness of the problem and does not require input parameters), there are many works for VIs with a more general class of operators, called H\"{o}lder-continuous, which covers both cases of operators with bounded variation and Lipschitz-continuous operators. VIs with H\"{o}lder-continuous monotone operators were considered in \cite{nemirovski2004prox}, where a special choice of the stepsize for the Mirror Prox algorithm led to the optimal complexity for this class of problems \cite{nemirovskij1983problem}. The authors of \cite{dang2015convergence} consider variational inequalities with non-monotone H\"{o}lder-continuous operators. Both papers use the H\"{o}lder constant and exponent to define the stepsize of the proposed methods. In \cite{bach2019universal}, for solving monotonic stochastic variational inequalities, authors proposed a universal method based on the proximal mirror method. Essentially, the standard proximal mirror method is used, in which the Lipshitz constant is proposed to be selected in a special way, similar to the method used in Adagrad. However, this method is not fully adaptive because, like Adagrad, the step selection strategy makes significant use of information about the dimension of the problem. A fully adaptive method for solving smooth stochastic monotonic variational inequalities was constructed (with minor reservations) in the work \cite{iusem2019variance} and in \cite{stonyakin2022generalized,stonyakin2021inexact} for the inexact oracle.

Recently, in \cite{rodomanov2024universal} the authors propose a universal gradient descent for stochastic convex optimization problems. In this paper, we propose an analog of this method for deterministic and stochastic variational inequalities. We propose new universal algorithms, called ''Universal Mirror Proximal Method'', for variational inequality problem \eqref{prob:VIweak}, in deterministic and stochastic settings.

Our main contributions can be summarized as follows:

\subsection{ Contributions}
\begin{itemize}
    \item We propose universal algorithms for the variational inequality problem in both deterministic and stochastic settings. These algorithms automatically adapt not only to the oracle’s noise (in the stochastic setting of the problem) but also to the H\"older continuity of the operator without having prior knowledge of either the problem class or the nature of the operator information. 

    \item We analyze the proposed algorithm for the problems in the deterministic setting with H\"older continuous operators and conclude its optimal convergence rate.

    \item We analyze the proposed stochastic algorithm for the problems with H\"older continuous unbiased operators and conclude its optimal convergence rate.

    \item We conduct some numerical experiments to test the proposed stochastic algorithm and compare its performance with other algorithms, such as SGD and Adam for the task of image classification.    
    
\end{itemize}

\subsection{Paper Organization}
The paper consists of an introduction and 3 main sections.  In Sect. \ref{sect:ProblemStatement}, we mentioned the statement of the considered problem, some examples that can be covered by the problem under consideration, and the connected fundamental concepts. In Sect. \ref{sect:UPM}, we proposed new algorithms called ''Universal Mirror Proximal Method (UMP)'' to solve the variational inequality problem in deterministic and stochastic settings. We analyze the proposed algorithms for the problems with H\"older continuous operators (and unbiased when the setting of the problem is stochastic) and conclude its optimal convergence rate.  Sect. \ref{sect:NumericalExperiments} devoted to the numerical experiments, to test the proposed stochastic algorithm and show its performance by comparison with other algorithms such as Stochastic Gradient Descent (SGD) and Adam for the task of image classification.     

\section{Preliminaries and Problem Statement}\label{sect:ProblemStatement}
We will use $\langle x, y \rangle := \sum_{i = 1}^{n} x_i y_i$ to denote the standard inner product of $x = (x_1, \ldots, x_n), y = (y_1, \ldots, y_n) \in \mathbb{R}^n$, and  $\|x\| = \sqrt{\langle x,x \rangle}$ to denote the Euclidean norm. The operator $\mathbb{E}[\cdot]$ denotes the full mathematical expectation. We also use the notation $|x|_+ := \max \{0, x\}$ for any $x \in \mathbb{R}$. 

Let $Q \subseteq \mathbb{R}^n$ be a compact and convex subset, bounded by a constant $D \ge \max_{x, y \in Q} \|x - y\|$. In this paper, we consider the following Minty variational inequality
\begin{equation}\label{prob:VIweak}
    \text{Find} \quad x^* \in Q : \quad \langle g(x), x^* - x \rangle \leq 0 \quad \forall x \in Q,
\end{equation}
where $g: Q \longrightarrow \mathbb{R}^n$ is a continuous, monotone operator, i.e., 
\begin{equation}\label{eq:CondMonotone}
    \langle g(x) - g(y) , x - y \rangle \geq 0 \quad \forall x, y \in Q.
\end{equation}

We called that the operator $g$ satisfies the H\"older condition on $Q$ (or $g$ is a H\"older-continuous operator), if for some $\nu \in [0,1]$  and $L_{\nu} \geq 0$, it holds the following inequality
\begin{equation}\label{eq:GHoeldDef}
\|g(x) - g(y)\| \leq L_{\nu}\|x-y\|^{\nu} \quad \forall x,y \in Q,
\end{equation}
where we  refer to $\nu$ as H\"older exponent and to $L_\nu$ as H\"older constant. 

Under the assumption of continuity and monotonicity of the operator $g$, the problem \eqref{prob:VIweak} is equivalent to a Stampacchia \cite{giannessi1998minty} (or strong \cite{nesterov2007dual}) variational inequality, in which the goal is to find $x^* \in Q $ such that 
\begin{equation}\label{prob:VIstrong}
    \langle g(x^*), x^* - x \rangle \leq 0, \quad \forall x \in Q.
\end{equation}

To emphasize the extensiveness of the problem \eqref{prob:VIweak} (or \eqref{prob:VIstrong}), 
Let us consider three common use cases for VIs.

\begin{example}[Minimization problem]\label{ex:minproblem}
    Let us consider the minimization problem 
    \begin{equation}\label{min_problem}
        \min_{x \in Q} f(x), 
    \end{equation}
    and assume that $g(x) = \nabla f(x)$, where $\nabla f(x)$ denotes to the (sub)gradient of $f$ at $x$. Then, if $f$ is convex, it can be proved that $x^* \in Q$  is a solution to \eqref{prob:VIstrong} if and only if $x^* \in Q$ is a solution to \eqref{min_problem}.
\end{example}

\begin{example}[Saddle point problem]\label{ex:saddleproblem}  
 Let us consider the saddle point problem
\begin{equation}\label{minmax_problem}
        \min_{u \in Q_u}\max_{v \in Q_v}  f(u, v), 
    \end{equation}
\end{example}
and assume that $g(x) : = g(u, v) = \left(\nabla_u f(u,v), -\nabla_v f(u, v)\right)^{\top}$, where $Q = Q_u \times Q_v$ with $Q_u \subseteq \mathbb{R}^{n_u}, Q_v \subseteq \mathbb{R}^{n_v}$. Then if $f$ is convex in $u$ and concave in $v$, it can be proved that $x^* \in Q$ is a solution to \eqref{prob:VIstrong} if and only if $x^* = (u^*, v^*) \in Q$ is a solution to \eqref{minmax_problem}. 

\begin{example}[Fixed point problem]\label{ex:fixedproblem}
Let us consider the fixed point
problem
\begin{equation}\label{fixed_prob}
    \text{find} \;\; x^* \in Q \;\; \text{such that} \quad F(x^*) = x^*,
\end{equation}
where $F: \mathbb{R}^n \longrightarrow \mathbb{R}^n $ is an operator. By taking $g(x)  = x - F(x)$, it can be proved that $x^* \in Q = \mathbb{R}^n$ is a solution to \eqref{prob:VIstrong} if $g(x^*) = 0$, i.e., $x^*$ is a solution to \eqref{fixed_prob}. 
\end{example}

Following \cite{nesterov2007dual,antonakopoulos2020adaptive}, to assess the quality of a candidate solution $\widehat{x}$, we use the following restricted  gap (or merit) function
\begin{equation}\label{eq:gap}
\operatorname{Gap}(\widehat{x}) = \max_{u \in Q} \langle g(u), \widehat{x}  - u \rangle .
\end{equation}

Thus, our goal is to find an approximate solution to the problem, that is, a point $\widehat{x} \in Q$ such that the following inequality holds
\begin{equation}\label{eq:Appr}
    \operatorname{Gap}(\widehat{x}) = \max_{u \in Q} \langle g(u), \widehat{x}  - u \rangle \leq \varepsilon,
\end{equation}
for some $\varepsilon >0$. 

As already mentioned, \cite{nemirovski2004prox} proposed Mirror Prox algorithm under the assumption of $L_1$-Lipschitz continuity of the operator, i.e., $g$ satisfies \eqref{eq:GHoeldDef} with $\nu=1$ and $L_1$. This method has complexity $O\left( L_1 D^2 / \varepsilon \right)$, where $D$ characterizes the diameter of the set $Q$ and $\varepsilon$ is the desired accuracy. By complexity we mean the number of iterations of an algorithm to find a point $\widehat{x} \in Q$ such that \eqref{eq:Appr} holds. For the case of variational inequalities with bounded variation of the operator $g$, i.e., $g$ satisfying \eqref{eq:GHoeldDef} with $\nu=0$ and $L_0$ \cite{nesterov2007dual} proposed a method with complexity $O\left( L_0^2 D^2 / \varepsilon^2 \right)$. The method for variational inequalities with H\"older-continuous monotone operator \cite{nemirovski2004prox} has the complexity $O\left(\left( L_{\nu} / \varepsilon\right)^{2/ (1+\nu)} D^2 \right),$
which is optimal for for the case of $\nu = 1$ and for the case of $\nu = 0$ \cite{nemirovskij1983problem,ouyang2021lower}.

For the stochastic setting of the problem \eqref{prob:VIweak} (or \eqref{prob:VIstrong}), we consider the following operator
\begin{equation}\label{stoc_g}
    g(x) = \mathbb{E}_{\xi } \left[g(x, \xi)\right],  
\end{equation}
where $\xi$ is a random variable, and $g(\cdot, \xi) : Q \longrightarrow \mathbb{R}^n$ is a stochastic operator, which satisfies the following assumption.

\begin{assumption}[Bounded variance]\label{assump_bv}
For all $\xi$ with any distribution, and $x \in Q$, the unbiased operator $g(\cdot, \xi)$ has uniformly bounded variance, i.e., it holds the following inequality
\begin{equation}
    \mathbb{E}_{\xi} \left[\|g(x, \xi) - g(x)\|^2\right] \le \sigma^2,
\end{equation}
for some $\sigma > 0$. 
\end{assumption}

\section{Universal proximal mirror method}\label{sect:UPM}

In this section, for problem \eqref{prob:VIweak}, we introduce new algorithms called Universal Mirror Proximal (UMP) for both deterministic and stochastic settings. For the deterministic setting, the algorithm is listed below as Algorithm \ref{alg:UMP}, whereas for the stochastic setting, the algorithm is listed as Algorithm \ref{alg:StocUMP}. These algorithms automatically adapt not only to the oracle’s noise (in the stochastic setting of the problem) but also to the H\"older continuity of the operator $g$ without having prior knowledge of either the problem class or the nature of the operator information, i.e., without knowing neither the H\"older exponent $\nu$, nor the corresponding H\"older constant $L_{\nu}$.

\subsection{Analysis for deterministic setting}

In this subsection, we provide an analysis for Algorithm \ref{alg:UMP}, when the setting of the problem \eqref{prob:VIweak} is deterministic.

\begin{algorithm}[htp]
\caption{Universal Mirror Proximal Method (UMP) (for the deterministic setting of the problem).}
\label{alg:UMP}
\textbf{Inputs:} $z_0 = \arg\min\limits_{u \in Q} d(u),\ L_0 = \|g(z_0)\|$, and diameter $D>0$.  
\hspace*{\algorithmicindent}
\begin{algorithmic}[1]
\For{$k = 0, 1,  \ldots$}
\State $w_k = \arg\min\limits_{x \in Q} \left(\left\langle g(z_k), x  - z_k \right\rangle  +  \frac{L_k}{2} \left\|z_k - x\right\|^2 \right),$ 
\State  $z_{k+1} = \arg\min\limits_{x \in Q} \left(\left\langle g(w_k), x - w_k\right\rangle  +  \frac{L_k}{2} \left\|z_k - x\right\|^2 \right),$
\State $L_{k+1} = L_k + \max\left\{0, \frac{2\left\langle g(w_k), w_k - z_{k+1}\right\rangle  - L_k \left \|z_k - z_{k+1}\right\|^2}{D^2 + \left \|z_k - z_{k+1}\right\|^2} \right\}.$
\EndFor
\end{algorithmic}
\end{algorithm}

At the first, let us prove the following lemma.

\begin{lemma}\label{lemma1}
Let $Q \subseteq \mathbb{R}^n$ be a convex set and $g : Q \longrightarrow \mathbb{R}^n$ be an operator. The function $h_k: Q \longrightarrow \mathbb{R}$, which defined on $Q$ as
\begin{equation}\label{eq:operator-h}
    h_k(x) = \left\langle g(w), x - w \right\rangle  +  \frac{L_k}{2}\|z - x\|^2 \quad \forall x, w, z \in Q,
\end{equation}
is $L_k$-strongly convex, for some $L_k > 0$.
\end{lemma}
\begin{proof}
Let $\alpha \in [0,1]$ and $x, y \in Q$. We have 
\begin{align*}
    h_k(\alpha x + (1 - \alpha) y) & = \left\langle g(w), \alpha x + (1 - \alpha) y - w \right\rangle  +  \frac{L_k}{2}\|\alpha x + (1 - \alpha) y - z\|^2 
    \\& = \alpha \left\langle g(w), x - w \right\rangle  + (1 - \alpha) \left\langle g(w), y - w \right\rangle  +  \frac{L_k}{2}\|\alpha x + (1 - \alpha) y - z\|^2
    \\& = \alpha h_k(x) + (1 - \alpha) h_k(y) +  \frac{L_k}{2}\|\alpha x + (1 - \alpha) y - z\|^2 - \alpha  \frac{L_k}{2}\|x - z\|^2 - (1 - \alpha)  \frac{L_k}{2}\|y - z\|^2
    \\& = \alpha h_k(x) + (1 - \alpha) h_k(y) + L_k \frac{\alpha^2 - \alpha}{2} \|x\|^2 + L_k \frac{(1 - \alpha)^2 - (1 - \alpha)}{2} \|y\|^2 + L_k \alpha(1 - \alpha) \left\langle x, y \right\rangle 
    \\& = \alpha h_k(x) + (1 - \alpha) h_k(y) - \frac{L_k \alpha(1 - \alpha)}{2}\|x - y\|^2.
\end{align*}
Thus, we get the following inequality
\[
    h_k \left(\alpha x + (1 - \alpha) y\right) \le \alpha h_k(x) + (1 - \alpha) h_k(y) - \frac{L_k \alpha(1 - \alpha)}{2}\|x - y\|^2. 
\]
This means that the function $h_k$ is $L_k$-strongly convex.
\end{proof}

Now, for $z_k$ and $w_k$ in Algorithm \ref{alg:UMP}, for every $k \geq 0$, let us set 
\begin{equation*}
  h_k(x) =  \left\langle g(w_k), x - w_k \right\rangle  +  \frac{L_k}{2} \|z_k - x\|^2 \quad \forall x \in Q.
\end{equation*}
Then, from item 3 in Algorithm \ref{alg:UMP}, we have
\begin{equation}\label{eq_000}
    z_{k+1} = \arg\min\limits_{x \in Q} h_k(x) \quad \text{for} \quad w = w_k,\ z = z_k,
\end{equation}
and, by Lemma \ref{lemma1} and \eqref{eq_000} we obtain the following inequality
\begin{equation}\label{eq_111}
    h_k(y) \ge h_k(z_{k+1}) + \frac{L_k }{2}\|y - z_{k+1}\|^2, \quad \forall y \in Q.
\end{equation}
Thus, for $w_k, z_k, \, \forall k \geq 0$ in Algorithm \ref{alg:UMP}, from \eqref{eq_111} for any $y \in Q$, we get
$$
    \left\langle g(w_k), y - w_k \right\rangle  +  \frac{L_k}{2}\|z_k - y\|^2 \ge \left\langle g(w_k), z_{k+1} - w_k \right\rangle  +  \frac{L_k}{2}\|z_k - z_{k+1}\|^2 +  \frac{L_k}{2}\|z_{k+1} - y\|^2,
$$
i.e., 
\begin{equation}\label{eq_222}
    \left\langle g(w_k), w_k - y\right\rangle  \le \left\langle g(w_k), w_k - z_{k+1} \right\rangle  -  \frac{L_k}{2}\|z_k - z_{k+1}\|^2 - \frac{L_k}{2}\|z_{k+1} - y\|^2 +  \frac{L_k}{2}\|z_k - y\|^2.
\end{equation}

By adding $ \frac{L_{k+1}}{2}\|z_{k+1} - y\|^2$ to the both sides of \eqref{eq_222}, we get
\begin{equation}\label{eq_2228}
\begin{aligned}
    \left\langle g(w_k), w_k - y \right\rangle  +  \frac{L_{k+1}}{2}\|z_{k+1} - y\|^2 & \le \left\langle g(w_k), w_k - z_{k+1} \right\rangle  - \frac{L_k }{2}\|z_k - z_{k+1}\|^2 +  \frac{L_{k+1} - L_k}{2}\|z_{k+1} - y\|^2
    \\& \quad +  \frac{ L_k}{2}\|z_k - y\|^2.
\end{aligned}
\end{equation}

Since $D \ge \max_{x, y \in Q} \|x - y\|$, and $L_{k+1} \ge L_k$ (see item 4 in Algorithm \ref{alg:UMP}) we get
\begin{align*}
    \left\langle  g(w_k), w_k - y \right\rangle  +  \frac{L_{k+1}}{2}\|z_{k+1} - y\|^2 & \le 
    \left\langle  g(w_k), w_k - z_{k+1} \right\rangle  -  \frac{L_k}{2}\|z_k - z_{k+1}\|^2 +  \left(L_{k+1} - L_k\right)\frac{ D^2}{2} 
    \\& \quad  +   \frac{L_k}{2}\|z_k - y\|^2.
\end{align*}
We will look for $L_{k+1}$ such that
\begin{equation}\label{eq_n1}
    \left(L_{k+1} - L_k\right) \frac{D^2}{2} = \left| \left\langle  g(w_k), w_k - z_{k+1} \right\rangle  -  \frac{L_k}{2}\|z_k - z_{k+1}\|^2\right|_+.
\end{equation}
Then, we get
\begin{equation}\label{eq_8585}
    \left\langle g(w_k), w_k - y \right\rangle  +  \frac{L_{k+1}}{2}\|z_{k+1} - y\|^2 \le D^2 \left(L_{k+1} - L_k\right) +  \frac{L_k}{2}\|z_k - y\|^2.
\end{equation}

Through the telescopic sum, we get
$$
    \sum\limits_{i=0}^k \left\langle g(w_i), w_i - y \right\rangle  +  \frac{L_{k+1}}{2}\|z_{k+1} - y\|^2 \le D^2 \left(L_{k+1} - L_0\right) + \frac{L_0}{2}\|z_0 - y\|^2.
$$
Thus, we have 
\begin{equation}\label{eq_4444}
    \frac{1}{k}\sum_{i=0}^k \left\langle g(w_i), w_i - y \right\rangle  \le \frac{D^2L_{k+1}}{k}.
\end{equation}

Let us set $\widehat{w} := \frac{1}{k}\sum_{i=0}^k w_i$, then we have
\begin{equation}\label{eq_5555}
    \operatorname{Gap}(\widehat{w}) = \max\limits_{y \in Q} \left\langle g(y), \widehat{w} - y \right\rangle  = \max\limits_{y \in Q} \frac{1}{k}\sum\limits_{i=0}^k \left( \left\langle g(w_i), w_i - y \right\rangle  + \left\langle g(y) - g(w_i), w_i - y \right\rangle  \right).
\end{equation}

Since $g$ is a monotone operator, then $\left\langle g(y) - g(w_i), w_i - y \right\rangle  \le 0$, and as a result from \eqref{eq_4444} and \eqref{eq_5555}, we get 
$$
    \operatorname{Gap}(\widehat{w}) \le \frac{D^2L_{k+1}}{k}.
$$

We have obtained an estimate of the convergence rate of the method for a given choice of $L_{k+1}$. 

Let us find an upper bound to $L_{k+1}$ to formulate the convergence rate depending on $k$. Using \eqref{eq_8585} we could not get a sufficient upper bound for $L_{k+1}$, thus instead \eqref{eq_n1} we take a slightly different formula, which allows us to estimate the upper bound for $L_{k+1}$ but we will get a slightly worse estimation for $\operatorname{Gap}(\widehat{w})$. This formula has the following form
\begin{equation}\label{eq_5858}
    \left(L_{k+1} - L_k\right)\frac{D^2 }{2}= \left|\left\langle g(w_k), w_k - z_{k+1} \right\rangle  -  \frac{L_{k+1}}{2}\|z_k - z_{k+1}\|^2\right|_+.
\end{equation}
Thus, for any $y \in Q$, we get
\begin{align*}
    \left\langle g(w_k), w_k - y\right\rangle  +  \frac{L_{k+1}}{2}\|z_{k+1} - y\|^2 & \le \left(D^2 + \|z_{k+1} - z_k\|^2\right) \left( L_{k+1} - L_k\right) +  \frac{L_k}{2}\|z_k - y\|^2 
    \\& \le 2D^2 \left(L_{k+1} - L_k\right) +  \frac{L_k}{2}\|z_k - y\|^2.
\end{align*}
From this, in a similar way as in the previous, we get the following estimate
\begin{equation}\label{eq:eatim_gap}
    \operatorname{Gap}(\widehat{w}) \le \frac{2D^2L_{k+1}}{k}.
\end{equation}

\begin{lemma}\label{lemma2}
Assume that the operator $g$ is monotone and  H\"older continuous with constant $L_{\nu}$ for some $\nu \in (0,1)$, given on a convex compact subset $Q$ bounded by a constant $D$. Then for problem (1), by Algorithm \ref{alg:UMP}, the following inequality holds
$$ 
    L_{k+1} \le \left(\frac{8k}{D^2}\right)^{\frac{1-\nu}{2}} L_{\nu}.
$$
\end{lemma}
\begin{proof} The recalculation of $L_{k+1}$ in Algorithm \ref{alg:UMP} is performed using \eqref{eq_5858}.  If $\left\langle g(w_k), w_k - z_{k+1} \right\rangle  -  \frac{L_{k+1}}{2}\|z_k - z_{k+1}\|^2  < 0$ (see \eqref{eq_5858}), then we get $L_{k+1} = L_k$, otherwise
\begin{equation}\label{eq_0101}
    \left(L_{k+1} - L_k\right)\frac{ D^2 }{2} = \left\langle g(w_k), w_k - z_{k+1} \right\rangle  -\frac{ L_{k+1} }{2}\|z_k - z_{k+1}\|^2.
\end{equation}
    
By using \eqref{eq:GHoeldDef}, we get 
$$
    \left\langle g(w_k) - g(z_k), w_k - z_{k+1} \right\rangle \le \left| \left\langle g(w_k) - g(z_k), w_k - z_{k+1} \right\rangle \right| \le L_{\nu}\|w_k - z_k\|^{\nu}\|w_k - z_{k+1}\|.
$$
Thus,
\begin{equation}\label{eq_2123}
    \left\langle g(w_k), w_k - z_{k+1} \right\rangle  \le \left\langle g(z_k), w_k - z_{k+1} \right\rangle  + L_{\nu}\|w_k - z_k\|^{\nu}\|w_k - z_{k+1}\|.
\end{equation}

From item 2 in Algorithm \ref{alg:UMP}, we have
\begin{equation}\label{eq_0008}
   w_k = \arg\min\limits_{x \in Q} h_k(x) \quad \text{for} \quad w = z = z_k,
\end{equation}
and, by Lemma \ref{lemma1} and \eqref{eq_0008} we obtain the following inequality
\begin{equation}\label{eq_1118}
    h_k(y) \ge h_k(w_k) + \frac{L_k }{2}\|y - w_k\|^2, \quad \forall y \in Q.
\end{equation}
Thus, for $w_k, z_k, \, \forall k \geq 0$ and $y = z_{k+1}$ in Algorithm \ref{alg:UMP}, from \eqref{eq_1118}, we get
$$
    \left\langle g(z_k), z_{k+1} - z_k \right\rangle  +  \frac{L_k}{2}\|z_k - z_{k+1}\|^2 \ge \left\langle g(z_k), w_k - z_k \right\rangle  +  \frac{L_k}{2}\|z_k - w_k\|^2 +  \frac{L_k}{2}\|z_{k+1} - w_k\|^2,
$$
i.e,
$$
\left\langle g(z_k), w_k - z_{k+1} + z_k - w_k \right\rangle  -  \frac{L_k}{2}\|z_k - z_{k+1}\|^2 \le \left\langle g(z_k), z_k - w_k \right\rangle  -  \frac{L_k}{2}\|z_k - w_k\|^2 -  \frac{L_k}{2}\|z_{k+1} - w_k\|^2.
$$
Thus, we get
\begin{equation}\label{eq_7575}
    \left\langle g(z_k), w_k - z_{k+1} \right\rangle  \le  \frac{L_k}{2}\|z_k - z_{k+1}\|^2 -  \frac{L_k}{2}\|z_k - w_k\|^2 -  \frac{L_k}{2}\|z_{k+1} - w_k\|^2.
\end{equation}

From \eqref{eq_2123} and \eqref{eq_7575}, we find
\begin{equation}\label{eq_666}
    \left\langle g(w_k), w_k - z_{k+1} \right\rangle  \le L_{\nu}\|w_k - z_k\|^{\nu}\|w_k - z_{k+1}\| +  \frac{L_k}{2}\|z_k - z_{k+1}\|^2 -  \frac{L_k}{2}\|z_k - w_k\|^2 -  \frac{L_k}{2}\|z_{k+1} - w_k\|^2.
\end{equation}
Thus, from \eqref{eq_0101} and \eqref{eq_666}, we get the following
$$
    \left(L_{k+1} - L_k\right)\frac{D^2}{2} \le L_{\nu}\|w_k - z_k\|^{\nu}\|w_k - z_{k+1}\| +  \frac{L_k - L_{k+1}}{2}\|z_k - z_{k+1}\|^2 -  \frac{L_k}{2}\|z_k - w_k\|^2 -  \frac{L_k}{2}\|z_{k+1} - w_k\|^2.
$$
Thus, 
\begin{align*}
    \frac{L_{k+1} - L_k}{2} \left(D^2 + \|z_k - z_{k+1}\|^2 - \|z_k - w_k\|^2 - \|z_{k+1} - w_k\|^2 \right) & \le L_{\nu}\|w_k - z_k\|^{\nu}\|w_k - z_{k+1}\| 
    \\& \quad -  \frac{L_{k+1}}{2}\left( \|z_k - w_k\|^2 + \|z_{k+1} - w_k\|^2 \right). 
\end{align*}

Let's denote $R = D^2 + \|z_k - z_{k+1}\|^2 - \|z_k - w_k\|^2 - \|z_{k+1} - w_k\|^2$. Then, we have
\begin{equation}\label{eq_03021}
    \frac{L_{k+1} - L_k }{2}R \le L_{\nu}\|w_k - z_k\|^{\nu}\|w_k - z_{k+1}\| -  \frac{L_{k+1}}{2}\left( \|z_k - w_k\|^2 + \|z_{k+1} - w_k\|^2 \right).
\end{equation}

Let us define the function $f(x,y) := L_{\nu}x^{\nu}y - \frac{1}{2}L_{k+1} (x^2 + y^2)$ on $\mathbb{R}^2$, which is concave for any $L_{\nu} > 0, L_{k+1} > 0$ and $\nu \in (0, 1)$. This function attains its maximal value at the point $(x, y ) \in \mathbb{R}^2$, such that 
$$
    \frac{\partial f}{\partial x} = \nu L_{\nu}x^{\nu - 1}y - L_{k+1} x = 0, \quad \text{and} \quad \frac{\partial f}{\partial y} = L_{\nu}x^{\nu} - L_{k+1} y = 0.
$$

From these equations, we get
$$
 y = \frac{L_{k+1} x^{2 - \nu}}{\nu L_{\nu}}  \quad \text{and} \quad  y = \frac{L_{\nu}x^{\nu}}{L_{k+1}}. 
$$
Thus, we have
$$
    \frac{L_{k+1} x^{2 - \nu}}{\nu L_{\nu}} = \frac{L_{\nu}x^{\nu}}{L_{k+1}} \quad  \Longrightarrow \quad x = \left(\frac{\sqrt{\nu} L_{\nu}}{L_{k+1}}\right)^{\frac{1}{1 - \nu}}, \quad 
    y = \left(\frac{\nu^{\nu/2} L_{\nu}}{L_{k+1}}\right)^{\frac{1}{1 - \nu}},
$$
and therefore we get
\begin{align}
    f(x, y) & \le L_{\nu}\left(\frac{\sqrt{\nu} L_{\nu}}{L_{k+1}}\right)^{\frac{\nu}{1 - \nu}}\left(\frac{\nu^{\nu/2} L_{\nu}}{L_{k+1}}\right)^{\frac{1}{1 - \nu}} - \frac{1}{2}L_{k+1} \left( \left(\frac{\sqrt{\nu} L_{\nu}}{L_{k+1}}\right)^{\frac{2}{1 - \nu}} + \left(\frac{\nu^{\nu/2} L_{\nu}}{L_{k+1}}\right)^{\frac{2}{1 - \nu}} \right) \nonumber
    \\& = \frac{{\nu}^{\frac{\nu}{1 - \nu}} L_{\nu}^{\frac{2}{1 - \nu}}}{L_{k+1}^{\frac{1 + \nu}{1 - \nu}}} \left(1 - \frac{\nu^{1/\nu}}{2} - \frac{1}{2} \right) = \frac{{\nu}^{\frac{\nu}{1 - \nu}}\left(1 - {\nu}^{1/\nu}\right) L_{\nu}^{\frac{2}{1 - \nu}}}{2 L_{k+1}^{\frac{1 + \nu}{1 - \nu}}}. \label{gtht}
\end{align}

By setting $x := \|w_k - z_k\|, y := \|w_k - z_{k+1}\|$ (in $f(x,y)$) and $p = \frac{1 + \nu}{1 - \nu}$,  then from \eqref{gtht} we find the following inequality 
$$
L_{\nu}\|w_k - z_k\|^{\nu}\|w_k - z_{k+1}\| -  \frac{L_{k+1}}{2}\left( \|z_k - w_k\|^2 + \|z_{k+1} - w_k\|^2 \right) \leq  \frac{{\nu}^{\frac{\nu}{1 - \nu}}\left(1 - {\nu}^{1/\nu}\right) L_{\nu}^{p+1}}{2 L_{k+1}^{p}}. 
$$
Thus, from \eqref{eq_03021} we find
$$
   \frac{L_{k+1} - L_k }{2}R \le \frac{{\nu}^{\frac{\nu}{1 - \nu}} \left(1 - {\nu}^{1/\nu}\right) L_{\nu}^{p+1}}{2 L_{k+1}^p},
$$
i.e.,
$$
    (p+1)L_{k+1}^p \left(L_{k+1} - L_k\right) \le \frac{{\nu}^{\frac{\nu}{1 - \nu}} \left(1 - {\nu}^{1/\nu}\right)(p+1) L_{\nu}^{p+1}}{R} := \alpha.
$$

Note that, since $L_{k+1} \geq L_k$, we have 
$$
    (p+1)L_{k+1}^p(L_{k+1} - L_k) \ge (p+1)\int_{L_k}^{L_{k+1}} t^{p} dt = L_{k+1}^{p+1} - L_k^{p+1}.
$$
Then
$$
    L_{k+1}^{p+1} - L_k^{p+1} \le \alpha,
$$
and from the telescopic sum, we get
\begin{equation}\label{eq_rxd}
    L_{k+1}^{p+1} \le k\alpha + L_0^{p+1}.
\end{equation}

For a sufficiently large $k$, we have $L_0^{p+1} \leq k\alpha$, then from \eqref{eq_rxd} with $p = \frac{1+\nu}{1-\nu}$, we get  
\begin{equation}\label{eq:estim_Lk1}
    L_{k+1} \le (2k\alpha)^{\frac{1}{p+1}} 
    = \left(2k \, \frac{{\nu}^{\frac{\nu}{1 - \nu}} \left(1 - {\nu}^{1/\nu}\right)(p+1) L_{\nu}^{p+1}}{R}\right)^{\frac{1}{p+1}} 
    = k^{\frac{1-\nu}{2}} \frac{{\nu}^{\nu/2} \left (1 - {\nu}^{1/\nu}\right)^{\frac{1-\nu}{2}} L_{\nu}}{\left(\frac{1 - \nu}{4}\right)^{\frac{1-\nu}{2}} R^{\frac{1-\nu}{2}}} %
    \le \left(\frac{4k}{R}\right)^{\frac{1-\nu}{2}} L_{\nu}.
\end{equation}

By the parallelogram identity $2 \|a\|^2 + 2\|b\|^2 = \|a+b\|^2 + \|a-b\|^2$ with $a = z_k - w_k, b = z_{k+1} - w_k $, we find
\begin{equation}\label{estim_R}
    R = D^2 + \|z_k - z_{k+1}\|^2 - \|z_k - w_k\|^2 - \|z_{k+1} - w_k\|^2 = D^2 + \frac{1}{2}\|z_k - z_{k+1}\|^2 - \frac{1}{2}\|z_k + z_{k+1} - 2w_k\|^2 \ge \frac{D^2}{2}.
\end{equation}
Thus, from \eqref{eq:estim_Lk1} and \eqref{estim_R}, we get the following upper bound of the $L_{k+1}$
\begin{equation*}
    L_{k+1} \le \left(\frac{8k}{D^2}\right)^{\frac{1-\nu}{2}} L_{\nu}.
\end{equation*}
\end{proof}

By combining \eqref{eq:eatim_gap} with the result of Lemma \ref{lemma2}, we have come to the following result, which indicates the convergence rate of the method in the deterministic case. 
\begin{theorem}
Assume that the operator $g$ is monotone and  H\"older continuous with constant $L_{\nu}$ for some $\nu \in (0,1)$, given on a convex compact subset $Q$ bounded by a constant $D$. Then for problem \eqref{prob:VIweak}, by Algorithm \ref{alg:UMP}, it holds the following inequality
$$
    \operatorname{Gap} (\widehat{w}) \le \frac{16 L_{\nu}D^{1+\nu}}{(8k)^{\frac{1+\nu}{2}}},
$$
where $\widehat{w} = \frac{1}{k}\sum_{i=0}^k w_i$. Consequently, to achieve an $\varepsilon$-solution to \eqref{prob:VIweak}, i.e., $\operatorname{Gap} (\widehat{w}) \leq \varepsilon $, it needs to make
$$
    k = O\left( \inf\limits_{\nu \in [0,1]} \left(\frac{16 L_{\nu}}{\varepsilon}\right)^{\frac{2}{1+\nu}} \frac{D^2}{8} \right)  = O\left( \inf\limits_{\nu \in [0,1]} \left(\frac{L_{\nu}}{\varepsilon}\right)^{\frac{2}{1+\nu}} 2^{\frac{5-3\nu}{1+\nu}} D^2 \right).
$$
oracle calls of Algorithm \ref{alg:UMP}.
\end{theorem}

\subsection{Analysis for stochastic setting}
In this subsection, we provide an analysis for Algorithm \ref{alg:StocUMP}, when the setting of the problem \eqref{prob:VIweak} is stochastic. Algorithm \ref{alg:StocUMP} is adaptive concerning the H\"older exponent $\nu$, H\"older constant $L_{\nu}$, and variance of the stochastic oracle $\sigma$. We mention that the adaptivity according to the parameter $\sigma$ is somewhat different from the adaptivity according to the parameter $L$. The adaptivity according to $\sigma$ can be understood in the sense that the $\sigma$ absences in the steps of Algorithm \ref{alg:StocUMP}, but it is not selected in the method as the selection of the parameter $L$ (see item 4 in Algorithm \ref{alg:StocUMP}). 

\begin{algorithm}[htp]
\caption{Universal Mirror Proximal Method (UMP) (for the stochastic setting of the problem).}
\label{alg:StocUMP}
\textbf{Inputs:} $z_0 = \arg\min\limits_{u \in Q} d(u),\ L_0 = \|g(z_0, \xi_{z_0})\|$ for a generated random variable $\xi_{z_0}$, and diameter $D>0$.
\hspace*{\algorithmicindent}
\begin{algorithmic}[1]
\For{$k = 0, 1,  \ldots$}
\State $w_k = \arg\min\limits_{x \in Q} \left(\left\langle g(z_k, \xi_{z_k}), x  - z_k \right\rangle  +  \frac{L_k}{2} \left\|z_k - x\right\|^2 \right),$ 
\State  $z_{k+1} = \arg\min\limits_{x \in Q} \left(\left\langle g(w_k, \xi_{w_k}), x - w_k\right\rangle  +  \frac{L_k}{2} \left\|z_k - x\right\|^2 \right),$
\State $L_{k+1} = L_k + \max\left\{0, \frac{2\left\langle g(w_k, \xi_{w_k}), w_k - z_{k+1}\right\rangle  - L_k \left \|z_k - z_{k+1}\right\|^2}{D^2 + \left \|z_k - z_{k+1}\right\|^2} \right\}.$
\EndFor
\end{algorithmic}
\end{algorithm}

We have $g(x) = \mathbb{E}_{\xi_x} [g(x, \xi_x)]$, but we can only get the recalculation of $L_{k+1}$, $w_k$ and $z_{k+1}$ using noisy operator $g(x, \xi_x)$.

Let us repeat all estimates from \eqref{eq:operator-h} to \eqref{eq_2228} for the operator $g(x, \xi_x)$ instead of $g(x)$ and obtain:
\begin{align*}
    \left\langle g(w_k, \xi_{w_k}), w_k - y \right\rangle  +  \frac{L_{k+1}}{2}\|z_{k+1} - y\|^2 & \le \left\langle g(w_k, \xi_{w_k}), w_k - z_{k+1} \right\rangle  - \frac{L_k }{2}\|z_k - z_{k+1}\|^2 
    \\& \quad +  \frac{L_{k+1} - L_k}{2}\|z_{k+1} - y\|^2 +  \frac{ L_k}{2}\|z_k - y\|^2.  
\end{align*}

We again look for $L_{k+1}$ such that:
$$ 
    \frac{\left(L_{k+1} - L_k\right) D^2 }{2} = \left|\left\langle g(w_k, \xi_{w_k}), w_k - z_{k+1}\right\rangle  -  \frac{L_{k+1}}{2}\|z_k - z_{k+1}\|^2\right|_+.
$$

Using this formula for $L_{k+1}$ we get
\begin{equation}\label{wq_kh214}
    \left\langle g(w_k, \xi_{w_k}), w_k - y\right\rangle  +  \frac{L_{k+1}}{2}\|z_{k+1} - y\|^2 \le 2D^2(L_{k+1} - L_k) +  \frac{L_k}{2}\|z_k - y\|^2.
\end{equation}

Since $\xi_{w_k}$ is independent with $w_k$, then $\mathbb{E}_{\xi_{w_k}} \left[\left\langle g(w_k, \xi_{w_k}), w_k - y\right\rangle \right]  = \left\langle g(w_k), w_k - y\right\rangle $. Thus,  by applying the expected value to both parts of the inequality \eqref{wq_kh214}, we get
$$
    \left\langle g(w_k), w_k - y\right\rangle  +  \mathbb{E}_{\xi_{w_k}} \left[ \frac{L_{k+1}}{2} \|z_{k+1} - y\|^2 \right] \le 2D^2 \left(\mathbb{E}_{\xi_{w_k}} \left[L_{k+1}\right] - L_k \right) + \frac{L_k}{2}\|z_k - y\|^2.
$$
and 
\begin{align*}
    \mathbb{E}_{\xi_{w_{k-1}}} \left[ \left\langle g(w_k), w_k - y\right\rangle \right]  +  \mathbb{E}_{\xi_{w_k}} \left[ \frac{L_{k+1}}{2} \|z_{k+1} - y\|^2 \right] & \le 2D^2 \left(\mathbb{E}_{\xi_{w_k}} \left[L_{k+1}\right] - \mathbb{E}_{\xi_{w_{k-1}}} \left[L_k\right] \right) 
    \\& \quad + \mathbb{E}_{\xi_{w_{k-1}}} \left[ \frac{L_k}{2}\|z_k - y\|^2 \right].
\end{align*}

Using the telescopic sum, we get
\begin{equation}\label{eq_4444stoc}
    \mathbb{E}_{\xi} \left[ \frac{1}{k}\sum_{i=0}^k \left\langle g(w_i), w_i - y \right\rangle \right]  \le \frac{2D^2 \mathbb{E}_{\xi} \left[L_{k+1}\right] }{k}.
\end{equation}

For $\widehat{w} = \frac{1}{k}\sum_{i=0}^k w_i$, we have
\begin{equation}\label{eq_5555stoc}
    \operatorname{Gap}(\widehat{w}) = \max\limits_{y \in Q} \left\langle g(y), \widehat{w} - y \right\rangle  = \max\limits_{y \in Q} \frac{1}{k}\sum\limits_{i=0}^k \left( \left\langle g(w_i), w_i - y \right\rangle  + \left\langle g(y) - g(w_i), w_i - y \right\rangle  \right).
\end{equation}
Since $g$ is a monotone operator, then $\left\langle g(y) - g(w_i), w_i - y \right\rangle  \le 0$,  and as a result from \eqref{eq_4444stoc} and \eqref{eq_5555stoc}, we get 
\begin{equation}\label{eq_estim_gapstoc}
    \mathbb{E}_{\xi} \left[\operatorname{Gap}(\widehat{w})\right] \le \frac{2D^2 \mathbb{E}_{\xi}\left[L_{k+1}\right]}{k}.
\end{equation}

Our goal now is finding an upper bound to $\mathbb{E}_{\xi}\left[L_{k+1}\right]$. To this goal, let us first prove the following lemma.

\begin{lemma}\label{lemma4}
Given a sequence of positive numbers $x_k, \alpha_k, \beta_k$. Let us assume that for any $k \geq 0$ there exists $p, q \ge 1$ such that at least one of the following two conditions is satisfied
$$
    x_{k+1}^p \le \alpha_k + x_k^p, \quad \text{and} \quad x_{k+1}^q \le \beta_k + x_k^q.  
$$
Then
\begin{equation}\label{eq_421bfgf}
    x_{k+1} \le \left(\sum\limits_{i=0}^k \alpha_i\right)^{1/p} + \left(\sum\limits_{i=0}^k \beta_i\right)^{1/q} + x_0.
\end{equation}
\end{lemma}
\begin{proof}
Let us prove \eqref{eq_421bfgf} by induction on $k \geq 0$.
\begin{itemize}
\item Let $x_1^p \le \alpha_0 + x_0^p$, then $x_1 \le \left(\alpha_0 + x_0^p\right)^{1/p} \le \alpha_0^{1/p} + \beta_0^{1/q} +  x_0 $.

\item Let $x_1^q \le \beta_0 + x_0^q$, then $x_1 \le \left(\beta_0 + x_0^q\right)^{1/q} \le \alpha_0^{1/p} + \beta_0^{1/q} + x_0$.
\end{itemize}

Let us assume that $x_{k} \le \left(\sum\limits_{i=0}^{k-1} \alpha_i\right)^{1/p} + \left(\sum\limits_{i=0}^{k-1} \beta_i\right)^{1/q} + x_0$. Then,

If  $x_{k+1}^p \le \alpha_k + x_k^p$, then we get
\begin{equation}\label{eq:47ghm}
    x_{k+1} \le \left(\alpha_k + \left(\left(\sum\limits_{i=0}^{k-1} \alpha_i\right)^{1/p} + \left(\sum\limits_{i=0}^{k-1} \beta_i\right)^{1/q} + x_0\right)^p \right)^{1/p}.
\end{equation}

Let us consider the vectors $(\alpha_k, 0, 0)$ and $\left(\sum_{i=0}^{k-1} \alpha_i, \left(\sum_{i=0}^{k-1} \beta_i \right)^{p/q}, x_0^p\right)$. Then, by the inverse triangle inequality for Minkowski space, we have
\begin{align*}
    \alpha_k + \left(\left(\sum\limits_{i=0}^{k-1} \alpha_i\right)^{1/p} + \left(\sum\limits_{i=0}^{k-1} \beta_i\right)^{1/q} + x_0\right)^p & = 
    \left(\alpha_k^{1/p} + 0^{1/p} + 0^{1/p}\right)^p + \left(\left(\sum\limits_{i=0}^{k-1} \alpha_i\right)^{1/p} + \left(\sum\limits_{i=0}^{k-1} \beta_i\right)^{p/q \ \cdot\ 1/p} + \left(x_0^p\right)^{1/p}\right)^p
     \\& 
     \le \left(\left(\alpha_k + \sum\limits_{i=0}^{k-1} \alpha_i\right)^{1/p} + \left( \sum\limits_{i=0}^{k-1} \beta_i\right)^{p/q \ \cdot\ 1/p} + \left(x_0^p\right)^{1/p}\right)^p.
\end{align*}
Thus, from \eqref{eq:47ghm}, we find
$$
    x_{k+1} \le \left(\sum\limits_{i=0}^{k} \alpha_i\right)^{1/p} + \left(0 + \sum\limits_{i=0}^{k-1} \beta_i\right)^{1/q} + x_0 \le \left(\sum\limits_{i=0}^{k} \alpha_i\right)^{1/p} + \left(\sum\limits_{i=0}^{k} \beta_i\right)^{1/q} + x_0.
$$

If $x_{k+1}^q \le \beta_k + x_k^q$, then we get
\begin{equation}\label{eq:oin12}
    x_{k+1} \le \left(\beta_k + \left(\left(\sum\limits_{i=0}^{k-1} \alpha_i\right)^{1/p} + \left(\sum\limits_{i=0}^{k-1} \beta_i\right)^{1/q} + x_0\right)^q \right)^{1/q}.
\end{equation}

Let us consider the vectors $(0, \beta_k, 0)$ and $\left( \left( \sum_{i=0}^{k-1} \alpha_i \right)^{q/p}, \sum_{i=0}^{k-1} \beta_i, x_0^q\right)$. Then, by the inverse triangle inequality for Minkowski space, we have
\begin{align*}
    \beta_k + \left(\left(\sum\limits_{i=0}^{k-1} \alpha_i\right)^{1/p} + \left(\sum\limits_{i=0}^{k-1} \beta_i\right)^{1/q} + x_0\right)^q
    &=
    \left(0^{1/q} + \beta_k^{1/q} + 0^{1/q}\right)^q + \left(\left(\sum\limits_{i=0}^{k-1} \alpha_i\right)^{q/p \ \cdot\ 1/q} + \left(\sum\limits_{i=0}^{k-1} \beta_i\right)^{1/q} + \left(x_0^q\right)^{1/q}\right)^q
    \\& \le 
    \left(\left(\sum\limits_{i=0}^{k-1} \alpha_i\right)^{q/p \ \cdot\ 1/q} + \left(\beta_k + \sum\limits_{i=0}^{k-1} \beta_i\right)^{1/q} + \left(x_0^q\right)^{1/q}\right)^q.
\end{align*}
Thus, from \eqref{eq:oin12}, we find
$$
    x_{k+1} \le \left(0 + \sum\limits_{i=0}^{k-1} \alpha_i\right)^{1/p} + \left(\sum\limits_{i=0}^{k} \beta_i\right)^{1/q} + x_0 \le \left(\sum\limits_{i=0}^{k} \alpha_i\right)^{1/p} + \left( \sum\limits_{i=0}^{k} \beta_i\right)^{1/q} + x_0.
$$
\end{proof}

Now, let us prove the following lemma, which gives an upper bound to $\mathbb{E}_{\xi}\left[L_{k+1}\right]$. 

\begin{lemma}\label{lemma3}
Assume that the operator $g$ is monotone and  H\"older continuous with constant $L_{\nu}$ for some $\nu \in (0,1)$, given on a convex compact subset $Q$ bounded by a constant $D$. Then for problem \eqref{prob:VIweak}, in the stochastic setting under Assumption \ref{assump_bv}, by Algorithm \ref{alg:StocUMP}, the following inequality holds
$$
    \mathbb{E}_{\xi} \left[L_{k+1}\right] \le 2\left(\frac{8k}{D^2}\right)^{\frac{1-\nu}{2}} L_{\nu} + \frac{\sqrt{8k}}{D} \sigma + L_0.
$$
\end{lemma}
\begin{proof}
If the right side is negative, then we get $L_{k+1} = L_k$, otherwise
\begin{align*}
    \frac{\left(L_{k+1} - L_k\right) D^2}{2} & = \left\langle g(w_k, \xi_{w_k}), w_k - z_{k+1} \right\rangle  -  \frac{L_{k+1}}{2}\|z_k - z_{k+1}\|^2
    \\& = \left\langle g(w_k, \xi_{w_k}) - g(w_k), w_k - z_{k+1} \right\rangle  + \left\langle g(w_k), w_k - z_{k+1} \right\rangle  -  \frac{L_{k+1}}{2}\|z_k - z_{k+1}\|^2.
\end{align*}

Similar to Lemma \ref{lemma2} from \eqref{eq_0101} to \eqref{eq_666}, we obtain
\begin{align*}
    \left(L_{k+1} - L_k\right)\frac{R}{2} & \le   L_{\nu}\|w_k - z_k\|^{\nu}\|w_k - z_{k+1}\| -  \frac{L_{k+1}}{2}\left( \|z_k - w_k\|^2 + \|z_{k+1} - w_k\|^2 \right) 
    \\& \quad + \left\langle g(w_k, \xi_{w_k}) - g(w_k), w_k - z_{k+1} \right\rangle .
\end{align*}

Let's denote $\delta_k := \|g(w_k, \xi_{w_k}) - g(w_k)\|$, then
$$
    \left( L_{k+1} - L_k \right) \frac{R}{2} \le L_{\nu} \| w_k - z_k \|^{\nu}\|w_k - z_{k+1} \| - \frac{ L_{k+1}}{2} \left( \|z_k - w_k\|^2 + \| z_{k+1} - w_k \|^2 \right) + \delta_k \| w_k - z_{k+1} \|.
$$

Let us define the function $f(x,y) := L_{\nu}x^{\nu}y + \delta_k y - \frac{L_{k+1}}{2} \left(x^2 + y^2\right)$ on $\mathbb{R}^2$, which is concave for any $L_{\nu} > 0, L_{k+1} > 0, \delta_k$ and $\nu \in (0, 1)$. This function attains its maximal value at the point $(x, y ) \in \mathbb{R}^2$, such that 
$$
    \frac{\partial f}{\partial x} = \nu L_{\nu}x^{\nu - 1}y - L_{k+1} x = 0, \quad \text{and} \quad L_{\nu}x^{\nu} - L_{k+1} y + \delta_k = 0.
$$

From these equations, we get
$$
    y = \frac{L_{k+1} x^{2 - \nu}}{\nu L_{\nu}}  \quad \text{and} \quad  y= \frac{L_{\nu}x^{\nu} + \delta_k}{L_{k+1}}. 
$$
Thus, we have
$$
    \frac{L_{k+1} x^{2 - \nu}}{\nu L_{\nu}} = \frac{L_{\nu}x^{\nu}}{L_{k+1}} + \frac{\delta_k}{L_{k+1}}.
$$

Let's substitute $x = C_k\left(\frac{\sqrt{\nu} L_{\nu}}{L_{k+1}}\right)^{\frac{1}{1 - \nu}}$, then $y = C_k^{2 - \nu} \left(\frac{\nu^{\nu/2} L_{\nu}}{L_{k+1}}\right)^{\frac{1}{1 - \nu}}$ and it holds the following inequality
$$
    \frac{L_{k+1} C_k^{2-\nu}}{\nu L_{\nu}} \left(\frac{L_{k+1}}{\nu^{\nu/ 2} L_{\nu}}\right)^{\frac{2 - \nu}{1 - \nu}} = \frac{L_{\nu} C_k^{\nu}}{L_{k+1}} \left(\frac{L_{k+1}}{\nu^{\nu/ 2} L_{\nu}}\right)^{\frac{\nu}{1 - \nu}} + \frac{\delta_k}{L_{k+1}}.
$$
Thus, 
$$
    C_k^{2-\nu} - C_k^{\nu} - \frac{\delta_k}{L_{k+1}}\left(\frac{L_{k+1}}{\nu^{\nu/ 2} L_{\nu}}\right)^{\frac{1}{1 - \nu}} = 0.
$$
Since $\delta_k \ge 0$, the maximum of the function $f(x, y)$ is achieved at $y \ge 0$, then $C_k \ge 0$.

Now let's substitute $x = C_k\left(\frac{\sqrt{\nu} L_{\nu}}{L_{k+1}}\right)^{\frac{1}{1 - \nu}}$ and $y = C_k^{2 - \nu} \left(\frac{\nu^{\nu/2} L_{\nu}}{L_{k+1}}\right)^{\frac{1}{1 - \nu}}$ in $f(x, y)$, then we get
\begin{align*}
    f(x, y) & = C_k^2 L_{\nu}\left(\frac{\sqrt{\nu} L_{\nu}}{L_{k+1}}\right)^{\frac{\nu}{1 - \nu}}\left(\frac{\nu^{\nu/ 2} L_{\nu}}{L_{k+1}}\right)^{\frac{1}{1 - \nu}} - \frac{L_{k+1}}{2} \left( C_k^2\left(\frac{\sqrt{\nu} L_{\nu}}{L_{k+1}}\right)^{\frac{2}{1 - \nu}} + C_k^{4-2\nu}\left(\frac{\nu^{\nu/2} L_{\nu}}{L_{k+1}}\right)^{\frac{2}{1 - \nu}} \right)
    \\& \quad  + \delta_k C_k^{2 - \nu} \left(\frac{\nu^{\nu/2} L_{\nu}}{L_{k+1}}\right)^{\frac{1}{1 - \nu}} 
    \\& = \frac{{\nu}^{\frac{\nu}{1 - \nu}} L_{\nu}^{\frac{2}{1 - \nu}}}{L_{k+1}^{\frac{1 + \nu}{1 - \nu}}} C_k^2 \left(1 - \frac{\nu^{1/\nu}}{2} - \frac{C_k^{2-2\nu}}{2} \right) + \delta_k C_k^{2 - \nu} \left(\frac{\nu^{\nu/2} L_{\nu}}{L_{k+1}}\right)^{\frac{1}{1 - \nu}}
    \\& = \frac{{\nu}^{\frac{\nu}{1 - \nu}} L_{\nu}^{\frac{2}{1 - \nu}}}{2L_{k+1}^{\frac{1 + \nu}{1 - \nu}}} C_k^2 \left(1 - \nu^{1/\nu} - \frac{\frac{\delta_k}{L_{k+1}} \left(\frac{L_{k+1}}{\nu^{\nu / 2} L_{\nu}}\right)^{\frac{1}{1 - \nu}}}{C_k^{\nu}}\right) + \delta_k C_k^{2 - \nu} \left(\frac{\nu^{\nu/2} L_{\nu}}{L_{k+1}}\right)^{\frac{1}{1 - \nu}} 
    \\& = \frac{{\nu}^{\frac{\nu}{1 - \nu}} L_{\nu}^{\frac{2}{1 - \nu}} \left(1 - \nu^{1/\nu}\right)}{2L_{k+1}^{\frac{1 + \nu}{1 - \nu}}} C_k^2 + \frac{\delta_k C_k^{2 - \nu}}{2} \left(\frac{\nu^{\nu/2} L_{\nu}}{L_{k+1}}\right)^{\frac{1}{1 - \nu}}.
\end{align*}

Let's consider two cases.
\begin{enumerate}
\item \textbf{First case:}
Let $\frac{\delta_k}{L_{k+1}} \left(\frac{L_{k+1}}{\nu^\frac{\nu}{2} L_{\nu}}\right)^{\frac{1}{1 - \nu}} \le C_k^{\nu}$, 
then 
$$
    0 = C_k^{2-\nu} - C_k^{\nu} - \frac{\delta_k}{L_{k+1}}\left(\frac{L_{k+1}}{\nu^\frac{\nu}{2} L_{\nu}}\right)^{\frac{1}{1 - \nu}} \ge C_k^{2-\nu} - 2C_k^{\nu}.
$$
Thus $C_k \le 2^{\frac{1}{2-2\nu}}$. This means that
\begin{align*}
    f(x, y) & = \frac{{\nu}^{\frac{\nu}{1 - \nu}} L_{\nu}^{\frac{2}{1 - \nu}} \left(1 - \nu^{1/\nu}\right)}{2L_{k+1}^{\frac{1 + \nu}{1 - \nu}}} C_k^2 + \frac{\delta_k C_k^{2 - \nu}}{2} \left(\frac{\nu^{\nu/2} L_{\nu}}{L_{k+1}}\right)^{\frac{1}{1 - \nu}}  
    \\& \le \frac{(2\nu)^{\frac{\nu}{1 - \nu}} L_{\nu}^{\frac{2}{1 - \nu}} \left(1 - \nu^{1/\nu}\right)}{L_{k+1}^{\frac{1 + \nu}{1 - \nu}}} + 2^{\frac{\nu}{1-\nu}} L_{k+1} \left(\frac{\nu^{\nu/2} L_{\nu}}{L_{k+1}}\right)^{\frac{2}{1 - \nu}}
    \\& = \frac{(2\nu)^{\frac{\nu}{1 - \nu}} L_{\nu}^{\frac{2}{1 - \nu}}}{L_{k+1}^{\frac{1 + \nu}{1 - \nu}}} \left(2 - \nu^{1/\nu}\right). 
\end{align*}

Let $p = \frac{1 + \nu}{1 - \nu}$, then, similarly to Lemma \ref{lemma2}, $L_{k+1}^{p+1} \le \alpha_k + L_k^{p+1}$, where
$$
    \alpha_k := \frac{2(p+1) (2\nu)^{\frac{\nu}{1 - \nu}} L_{\nu}^{p+1}}{R}\left(2 - \nu^{1/\nu}\right).
$$ 

\item \textbf{Second case:}
Let $\frac{\delta_k}{L_{k+1}} \left(\frac{L_{k+1}}{\nu^{\nu/ 2} L_{\nu}}\right)^{1/(1 - \nu)} \ge C_k^{\nu}$, 
then 
$$
    0 = C_k^{2-\nu} - C_k^{\nu} - \frac{\delta_k}{L_{k+1}}\left(\frac{L_{k+1}}{\nu^{\nu/ 2} L_{\nu}}\right)^{1/(1 - \nu)} \ge C_k^{2-\nu} - \frac{2\delta_k}{L_{k+1}}\left(\frac{L_{k+1}}{\nu^{\nu/ 2} L_{\nu}}\right)^{1/(1 - \nu)},
$$ 
therefore $C_k \le \left(\frac{2\delta_k}{L_{k+1}}\left(\frac{L_{k+1}}{\nu^{\nu/2} L_{\nu}}\right)^{1/(1 - \nu)}\right)^{1/(2-\nu)}$. This means that
\begin{align*}
    f(x, y) & = \frac{{\nu}^{\nu/(1 - \nu)} L_{\nu}^{2/(1 - \nu)} \left(1 - \nu^{1/\nu}\right)}{2L_{k+1}^{\frac{1 + \nu}{1 - \nu}}} C_k^2 + \frac{\delta_k C_k^{2 - \nu}}{2} \left(\frac{\nu^{\nu/2} L_{\nu}}{L_{k+1}}\right)^{1/(1 - \nu)} 
    \\& \le \frac{\delta_k C_k^{2 - \nu} \left(2 - \nu^{1/\nu}\right)}{2} \left(\frac{\nu^{\nu/2} L_{\nu}}{L_{k+1}}\right)^{1/(1 - \nu)}
    \\& \le \frac{\delta_k^2 \left(2 - \nu^{1/\nu}\right)}{L_{k+1}}.
\end{align*}

Let $q = 1$, then, similarly to Lemma \ref{lemma2}, $L_{k+1}^{q+1} \le \beta_k + L_k^{q+1}$, where
$$
    \beta_k := \frac{2\delta_k^2 \left(2 - \nu^{1/ \nu}\right)}{R}.
$$
\end{enumerate}

By Lemma \ref{lemma4} we get
\begin{align*}
L_{k+1} & \le \left(k \frac{2(p+1) (2\nu)^{\frac{\nu}{1 - \nu}} L_{\nu}^{p+1}}{R}\left(2 - \nu^{1/\nu}\right)\right)^{\frac{1}{p+1}} + \sqrt{\frac{2\left(2 - \nu^{1/\nu}\right)}{R} \sum\limits_{i=1}^k \delta_i^2} + L_0 
\\& \le 2^{\frac{2-\nu}{2}} k^{\frac{1-\nu}{2}} \frac{{\nu}^{\nu/2} L_{\nu}}{(1 - \nu)^{\frac{1-\nu}{2}} R^{\frac{1-\nu}{2}}} \left(2 - \nu^{1/\nu}\right)^{\frac{1-\nu}{2}} + \frac{2}{\sqrt{R}} \sqrt{\sum\limits_{i=1}^k \delta_i^2} + L_0
\\& \le 2\left(\frac{4k}{R}\right)^{\frac{1-\nu}{2}} L_{\nu} + \frac{2}{\sqrt{R}} \sqrt{\sum\limits_{i=1}^k \delta_i^2} + L_0.
\end{align*}

According to Jensen's inequality, we have $\mathbb{E}_{\xi} \left[ \left(\sum_{i=1}^k \delta_i^2\right)^{1/2}\right] \le \left(\sum_{i=1}^k \mathbb{E}_{\xi} \left[\delta_i^2\right]\right)^{1/2} = \sigma\sqrt{k}$. 
Thus, we get the following
$$
    \mathbb{E}_{\xi} \left[L_{k+1}\right] \le 2\left(\frac{8k}{D^2}\right)^{\frac{1-\nu}{2}} L_{\nu} + \frac{\sqrt{8k}}{D} \sigma + L_0.$$
\end{proof}

By combining \eqref{eq_estim_gapstoc} with the result of Lemma \ref{lemma3}, we have come to the following result, which indicates the convergence rate of the method in the stochastic case. 
\begin{theorem}
Assume that the operator $g$, as in \eqref{stoc_g}, is monotone and  H\"older continuous with constant $L_{\nu}$ for some $\nu \in (0,1)$, given on a convex compact subset $Q$ bounded by a constant $D$.  Then for problem \eqref{prob:VIweak}, in the stochastic setting under Assumption \ref{assump_bv}, by Algorithm \ref{alg:StocUMP}, it holds the following inequality
$$
    \mathbb{E}_{\xi} \left[\operatorname{Gap} (\widehat{w}) \right] \le 32\left(\frac{D^2}{8k}\right)^{\frac{1+\nu}{2}} L_{\nu} + \frac{32 D}{\sqrt{8k}} \sigma + \frac{2D^2 L_0}{k},
$$
where $\widehat{w} = \frac{1}{k}\sum_{i=0}^k w_i$ and $L_0$ is an initial guess of the H\"older constant. 
\end{theorem}

\section{Numerical experiments}\label{sect:NumericalExperiments}
To compare the performance of Algorithm \ref{alg:StocUMP} (UMP) with other optimizers, let’s train a convolutional neural network \href{https://pytorch.org/vision/main/models/generated/torchvision.models.resnet50.html}{resnet18} implemented in the PyTorch library based on articles \cite{he2016deep} for image classification dataset \href{https://pytorch.org/ vision/main/generated/torchvision.datasets.CIFAR10.html}{CIFAR10} also provided by the PyTorch library.

CIFAR10 is a dataset of object samples from 10 classes (airplane, car, bird, cat, deer, dog, frog, horse, ship, and truck), containing 50,000 train and 10,000 test color pictures of size $32 \times 32$, each signed by the corresponding class.

\begin{figure}[htbp]
\begin{minipage}[h]{\linewidth}
\centering
\includegraphics[width=0.9\linewidth, height=7cm]{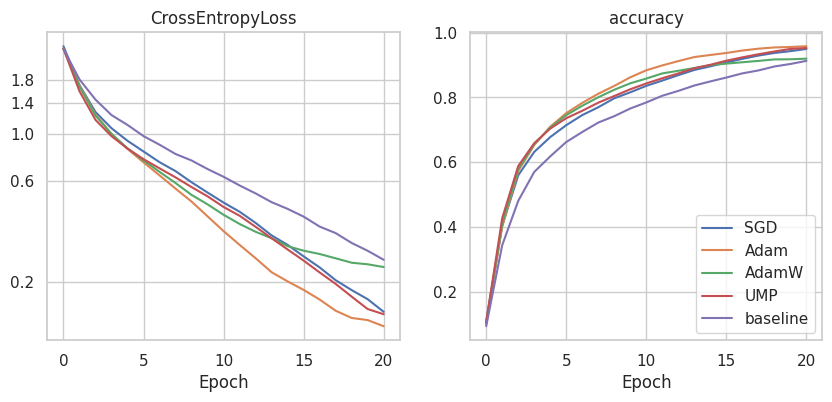}%
\captionsetup{justification=centering}
\caption{Comparison of the optimizers metrics for training when training resnet18 on the CIFAR10 dataset.}
\label{fig:MNIST}
\end{minipage}
\hfill
\end{figure}

\begin{figure}[htbp]
\begin{minipage}[h]{\linewidth}
\centering
\includegraphics[width=0.9\linewidth, height=7cm]{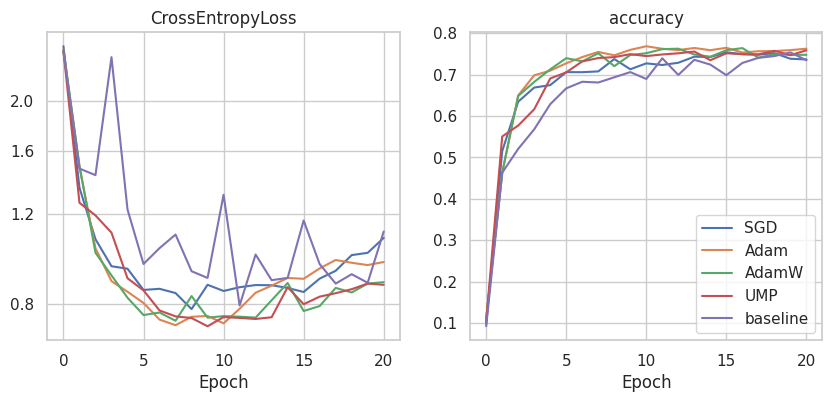}%
\captionsetup{justification=centering}
\caption{Comparison of the optimizers metrics for validation when training resnet18 on the CIFAR10 dataset.}
\label{fig:CIFAR10}
\end{minipage}
\hfill
\end{figure}

As optimizers, we consider SGD, Adam, AdamW, UMP (Algorithm \ref{alg:StocUMP}), and UMP without changing the variable $L$ as a baseline. We select the optimal learning rate for each optimizer and our method uses it as initialization for variable $L$.

The most important graph is the comparison of the optimizers losses for training when training resnet18 on the CIFAR10 dataset. It shows the speed of the CrossEntropyLoss minimization which is the problem that the optimizers solve in the experiment. We see that compared to the baseline, changing the variable $L$ in the UMP method improves the speed of the loss minimization considerably. We also see that the UMP achieves results similar to SGD for training and similar to Adam and AdamW for validation when training resnet18 on the CIFAR10 dataset.

\section{Conclusion}\label{sect:Conclusion}

In this paper, we show known improvements to the proximal mirror method \cite{nemirovski2004prox} and propose our universal version in both deterministic and stochastic settings of the problem, based on the ideas of universal gradient descent  \cite{rodomanov2024universal}. We prove the theoretical convergence of the proposed universal proximal mirror methods and obtain an estimate of the convergence rate in the deterministic case
$$
    O\left( \inf\limits_{\nu \in [0,1]} \left(\frac{L_{\nu}}{\varepsilon}\right)^{\frac{2}{1+\nu}} 2^{\frac{5-3\nu}{1+\nu}} D^2 \right).
$$

In the stochastic case, we obtain the following estimate for the  convergence rate
$$
    \mathbb{E}_{\xi} \left[\operatorname{Gap} \left(\frac{1}{k}\sum\limits_{i=0}^k w_i \right) \right] \le 32\left(\frac{D^2}{8k}\right)^{\frac{1+\nu}{2}} L_{\nu} + \frac{32 D}{\sqrt{8k}} \sigma + \frac{2D^2 L_0}{k}.
$$

We also compare the stochastic method's performance with popular optimizers on the classical minimization problem. The experiment shows that the proposed method achieves similar results for loss minimization.

In the future, we plan to conduct more experiments with other models and compare the performance of the universal proximal mirror methods with different methods for saddle-point problems, namely the problem of training generative adversarial networks. It is also planned to study the behavior of the proposed methods with restarts to accelerate its convergence rate for variational inequalities with strongly monotone operators. Generalize the proposed algorithms and the conducted analysis to be convenient with any norm (i.e., with an arbitrary prox function and Bregman divergence) not necessarily Euclidean. Also, we plan to investigate adaptive batch size strategies (see eg. \cite{lau2024adadagrad}) to solve the variational inequality problem in the stochastic setting for more general operators.

\subsection*{Acknowledgements}

We thank Abhishek for finding a mistake in Lemmas \ref{lemma2} and \ref{lemma3} which are proven for $\nu \in (0, 1)$, not $\nu \in [0, 1]$. This work was supported by a grant for research centers in the field of artificial intelligence, provided by the Analytical Center for the Government of the Russian Federation in accordance with the subsidy agreement (agreement identifier 000000D730324P540002) and the agreement with the Moscow Institute of Physics and Technology dated November 1, 2021 No. 70-2021-00138.

\bibliographystyle{cas-model2-names}

\bibliography{cas-refs}

\end{document}